\documentclass[12pt]{amsproc}
\newtheorem{theorem}{\sc Theorem}[section]
\newtheorem{lemma}[theorem]{\sc Lemma}
\newtheorem{proposition}[theorem]{\sc Proposition}
\newtheorem{corollary}[theorem]{\sc Corollary}

\begin{document}

\title[Engel-like conditions]{Orderable groups with Engel-like conditions}

\author{Pavel Shumyatsky }
\address{ Department of Mathematics, University of Brasilia,
Brasilia-DF, 70910-900 Brazil}
\email{pavel@unb.br}

\thanks{This research was supported by FAPDF and CNPq-Brazil}

\begin{abstract}
Let $x$ be an element of a group $G$. For a positive integer $n$ let $E_n(x)$ be the subgroup generated by all commutators  $[...[[y,x],x],\dots,x]$ over $y\in G$,  where $x$ is repeated $n$ times. There are several recent results showing that certain properties of groups with small subgroups $E_n(x)$ are close to those of Engel groups. The present article deals with orderable groups in which, for some $n\geq1$, the subgroups $E_n(x)$ are polycyclic. Let $h,n$ be positive integers and $G$ an orderable group in which $E_n(x)$ is polycyclic with Hirsch length at most $h$ for every $x\in G$. It is proved that there are $(h,n)$-bounded numbers $h^*$ and $c^*$ such that $G$ has a finitely generated normal nilpotent subgroup $N$ with $h(N)\leq h^*$ and $G/N$ nilpotent of class at most $c^*$.
\end{abstract}
\maketitle

\section{Introduction} A group $G$ is called an Engel group if for every $x,y\in G$ the equation $[y,x,x,\dots,x]=1$ holds, where $x$ is repeated in the commutator sufficiently many times depending on $x$ and $y$. Throughout the paper, we use the left-normed simple commutator notation $[a_1,a_2,a_3,\dots,a_r]=[...[[a_1,a_2],a_3],\dots,a_r]$. The long commutators $[y,x,\dots,x]$, where $x$ occurs $i$ times, are denoted by $[y,{}_i\,x]$. An element $x\in G$ is $n$-Engel if $[y,{}_n\,x]=1$ for all $y\in G$. A group $G$ is $n$-Engel if $[y,{}_n\,x]=1$ for all $x,y\in G$. Given $x\in G$, the subgroup $E_n(x)$ is the one generated by all elements of the form $[y,{}_n\,x]$ where $y$ ranges over $G$. Note that $E_n(x)$ is not the same as the more familiar subnormal subgroup $[G,{}_n\,x]=[[G,{}_{n-1}\,x],x]$. There are several recent results showing that certain properties of groups with small subgroups $E_n(x)$ are close to those of Engel groups (see for instance \cite{khushu1,khushu2,mona}). The present article deals with orderable groups. A group $G$ is called orderable if there exists a full order relation $\leq$ on the set $G$ such that $x\leq y$ implies $axb\leq ayb$ for all $a,b,x,y\in G$, i.e. the order on $G$ is compatible with the product of $G$. Kim and Rhemtulla proved that any orderable $n$-Engel group is nilpotent (\cite{kr2}, see also \cite{LM}). In the present article we consider orderable groups $G$ such that the subgroup $E_n(x)$ is polycyclic for each $x\in G$. Recall that a group is polycyclic if and only if it admits a finite subnormal series all of whose factors are cyclic. The Hirsch length $h(K)$ of a polycyclic group $K$ is the number of infinite factors in the subnormal series.

Our aim here is to prove the following theorem.

\begin{theorem}\label{main} Let $h,n$ be positive integers and $G$ an orderable group in which $E_n(x)$ is polycyclic with $h(E_n(x))\leq h$ for every $x\in G$. There are $(h,n)$-bounded numbers $h^*$ and $c^*$ such that $G$ has a finitely generated normal nilpotent subgroup $N$ with $h(N)\leq h^*$ and $G/N$ nilpotent of class at most $c^*$. 
\end{theorem}
Note that if $h=0$, the proof shows that $N=1$ and our result becomes the theorem of Kim and Rhemtulla.

One tool used in the proof of Theorem \ref{main} deserves a special mention. A well-known theorem of Malcev states that a soluble group of automorphisms of a polycyclic-by-finite group is polycyclic \cite{malcev}. We require the following quantitative variation of Malcev's theorem.
\medskip

{\em Let $N$ be a polycyclic-by-finite group with $h(N)=h$ and let $\Gamma$ be a soluble group of automorphisms of $N$. Then $h(\Gamma)<h^2+2h$. 
}\medskip

To our surprise, in the literature we did not find any mention of the fact that $h(\Gamma)$ should be bounded in terms of $h$ and so it seems that this has so far gone unnoticed. The author is grateful to Dan Segal for suggesting the proof given here (see Proposition \ref{polyc} in Section 3).

\section{Preliminaries}

We start with general facts about nilpotent groups and Engel elements. If $\alpha$ is an automorphism (or just an element) of a group $G$, the subgroup generated by the elements of the form $g^{-1}g^\alpha$ with $g\in G$ is denoted by $[G,\alpha]$. It is well-known that the subgroup $[G,\alpha]$ is an $\alpha$-invariant normal subgroup in $G$. Throughout, we write $[G,{}_i\,\alpha]$ for $[[G,{}_{i-1}\,\alpha],\alpha]$.

\begin{lemma}\label{aa} Let $G=H\langle a\rangle$ where $H$ is a nilpotent of class $c$ normal subgroup and $a$ is an $n$-Engel element. Then $G$ is nilpotent with class at most $cn$.
\end{lemma}
\begin{proof} Let $K=Z(H)$ and set $K_0=K$ and $K_{i+1}=[K_i,a]$ for $i=0,1,\dots$. Then $K_{n-1}\leq K\cap C_K(a)$ and so $K_{n-1}\leq Z(G)$. Moreover we observe that $[K_{i-1},G]\leq K_i$ and it follows that $K_{n-i}\leq Z_{i}(G)$ for $i=1,2,\dots,n$. Therefore $K\leq Z_{n}(G)$. Passing to the quotient $G/Z_{n}(G)$ and using induction on $c$ we deduce that $G$ is nilpotent with class at most $cn$.
\end{proof}

We write $\gamma_i(G)$ for the $i$th term of the lower central series of a group $G$.

\begin{lemma}\label{bb} For any positive integers $c,n$ there exists an integer $f=f(c,n)$ with the following property. Let $G=H\langle a\rangle$ where $H$ is a nilpotent of class $c$ normal subgroup. Then $\gamma_f(G)\leq E_n(a)$.
\end{lemma}
\begin{proof} Fix $n\geq1$ and use induction on $c$. If $H$ is abelian, we obviously have $\gamma_{n+1}(G)\leq E_n(a)$ and so it is enough to choose $f=n+1$. Assume that $c\geq2$ and let $Z=Z(H)$. By induction there exists a bounded number $s$ such that $\gamma_s(G)\leq ZE_n(a)$. Let $E= E_n(a)\cap Z\gamma_s(G)$. So $ZE$ is normal in $G$ and $\gamma_s(G)\leq ZE$. Set $Z_0=ZE$ and for $i=0,1,\dots,s-1$ let $Z_i$ denote the full inverse image of $Z_i(G/Z_0)$. Further, for $i=0,1,\dots,s-1$ we set $G_i=Z_i\langle a\rangle$. It is clear that $Z_{s-1}=G_{s-1}=G$.

Since $Z$ is abelian, $[Z,{}_na]\leq E$. We observe that $Z$ and $E$ are commuting $a$-invariant subgroups and so $[Z_0,{}_na]\leq E$. Let $T$ be the normal closure of $[Z_0,{}_na]$ in $G_0$. It is clear that $T\leq E$. Since the image of $a$ in $G_0/T$ is $n$-Engel, Lemma \ref{aa} implies that there exists a bounded number $k$ such that $G_0/T$ is nilpotent with class at most $k-1$ and so $\gamma_k(G_0)\leq E$.

By induction on $i$ we will show that there exists a bounded number $k_i$ such that $\gamma_{k_i}(G_i)\leq E$. Once this is done, we will simply set $f=k_{s-1}$. Assume that for some $j\leq s-1$ there exists $k_j$ with the property that  $\gamma_{k_j}(G_j)\leq E$. If $j=s-1$ we have nothing to prove so we suppose that $j\leq s-2$. Since $G_{j+1}$ normalizes $G_j$, it follows that $\gamma_{k_j}(G_j)$ is normal in $G_{j+1}$. Recall that $\gamma_s(G)\leq G_0$. It follows that the image of $a$ in $G_{j+1}/\gamma_{k_j}(G_j)$ is $(s+k_j)$-Engel, whence by Lemma \ref{aa} the factor-group $G_{j+1}/\gamma_{k_j}(G_j)$ is nilpotent with bounded class, say $k_{j+1}$. We conclude that $\gamma_{k_{j+1}}(G_{j+1})\leq E$. This completes the proof.
\end{proof}

The next lemma is rather obvious so the proof is omitted.
\begin{lemma}\label{ni1} Let $G$ be a finitely generated nilpotent group and $H$ a subgroup of finite index in $G$. Then $H'$ has finite index in $G'$.
\end{lemma}

\begin{lemma}\label{ni2} Let $G$ be a finitely generated nilpotent group and $\phi$ an automorphism of $G$ such that $[G,\phi]$ has finite index in $G$. Then $[G,\phi,\phi]$ has finite index in $G$ as well.
\end{lemma}
\begin{proof} We use induction on the nilpotency class of $G$. Suppose first that $G$ is abelian and let $m$ be a positive integer such that $G^m\leq[G,\phi]$. Then $[G,\phi,\phi]$ contains $[G^m,\phi]=[G,\phi]^m$ which has finite index in $[G,\phi]$.

Now suppose that $G$ is non-abelian and both subgroups $[G,\phi,\phi]G'$ and $[G,\phi,\phi]Z(G)$ have finite index in $G$. By Lemma \ref{ni1},  $[G,\phi,\phi]'= ([G,\phi,\phi]Z(G))'$ has finite index in $G'$. This implies that $[G,\phi,\phi]$ has finite index in $[G,\phi,\phi]G'$, whence the lemma follows.
\end{proof}
\begin{corollary}\label{ni3}  Let $G=H\langle a\rangle$ be a nilpotent group with a normal torsion-free subgroup $H$ of Hirsch length $h$. Then $G$ is nilpotent with $h$-bounded class.
\end{corollary}
\begin{proof} We assume that $h\geq1$. It is clear that $H$ has nilpotency class at most $h-1$. In view of Lemma \ref{aa} we need to show that $a$ is $n$-Engel for some $h$-bounded number $n$. Lemma \ref{ni2} implies that whenever $[H,{}_i\,a]$ is infinite the subgroup $[H,{}_{i+1}\,a]$ has infinite index in $[H,{}_i\,a]$. Therefore whenever $[H,{}_i\,a]$ is infinite, $h([H,{}_{i+1}\,a])<h([H,{}_i\,a])$. Hence, $a$ is $n$-Engel with $n\leq h$.
\end{proof}

Given subgroups $X$ and $Y$ of a group $G$, we denote by $X^Y$ the smallest subgroup of $G$ containing $X$ and normalized by $Y$. We say that a group $G$ satisfies $max$ if $G$ satisfies the maximal condition on subgroups.

\begin{lemma}\label{x^y} Let $x$ and $y$ be elements of a group $G$ and suppose that for some $n\geq 1$ the subgroup $E_n(y)$ satisfies $max$. Then $\langle x\rangle^{\langle y\rangle}$ is finitely generated.
\end{lemma}

\begin{proof} Observe that ${\langle x\rangle}^{\langle y\rangle}$ is generated by all commutators $[x,{}_iy]$ with $i=0,1,\dots$. Set $X=\langle x\rangle^{\langle y\rangle}\cap E_n(y)$. We have $${\langle x\rangle}^{\langle y\rangle}=\langle x,[x,y],\dots,[x,{}_{n-1}y],X\rangle.$$ Since  $E_n(y)$ satisfies $max$, $X$ is finitely generated and so the lemma follows.
\end{proof}

\begin{corollary}\label{H^y}
Let $y$ be an element of a group $G$ and $H$ a finitely generated subgroup. Suppose that for some $n\geq 1$ the subgroup $E_n(y)$ satisfies $max$. Then $H^{\langle y\rangle}$ is finitely generated.
\end{corollary}

The following lemma is well-known. We supply the proof for the reader's convenience.

\begin{lemma}\label{comm}
If $G$ is a group generated by two elements $x$ and $y$, then $G'=\langle [x,y]^{x^ry^s}\,|\,r,s\in\mathbb{Z}\rangle$.
\end{lemma}

\begin{proof}
Let $N=\langle [x,y]^{x^ry^s}\,|\,r,s\in\mathbb{Z} \rangle$. Of course, $N^y$ and $N^{y^{-1}}$ are both contained in $N$. Moreover,
$$[x,y]^{x^r y^s x}=[x,y]^{x^{r+1}y^s [y^s,x]}=[y^s,x]^{-1}[x,y]^{x^{r+1}y^s}[y^s,x].$$
We have $[y^s,x]=[y,x]^{y^{s-1}}[y,x]^{y^{s-2}}\cdots[y,x]$, for all $s\geq 1$. This implies that $N^x\leq N$. Similarly we get $N^{x^{-1}}\leq N$ and so $N$ is normal in $G$. It follows that $G'=N$, as desired.
\end{proof}

\begin{lemma}\label{BM} Let $n\geq 1$ and $G$ be a group generated by a finite set $Y$ such that $E_n(y)$ satisfies $max$ for all $y\in Y$. Then $G'$ is finitely generated.
\end{lemma}

\begin{proof} First assume that $Y=\{x,y\}$. Then $G'=\langle [x,y]^{x^ry^s}\,|\,r,s\in\mathbb{Z}\rangle \rangle$ by Lemma \ref{comm} and we are done since $(\langle [x,y]\rangle^{\langle x\rangle})^{\langle y\rangle}$ is finitely generated by Corollary \ref{H^y}. Now suppose that $Y=\{y_1,\dots,y_{d}\}$ with $d\geq 3$, and assume that the result is true for subgroups which can be generated by at most $d-1$ elements from $Y$. For $i=1,\dots,d$ set $G_i=\langle y_1,\dots,y_{i-1},y_{i+1},\dots,y_{d}\rangle$.
The induction hypothesis yields that $G_i'$ is finitely generated and, by Corollary \ref{H^y}, the same is true for $(G_i')^{\langle y_i\rangle}$. It is easy to see that $K=\langle(G_i')^{\langle y_i\rangle}\,|\,i=1,\dots,d\rangle$ is a normal subgroup of $G$ and hence $G'=K$. In particular, $G'$ is finitely generated.
\end{proof}

Now, an easy induction gives us the following corollary.

\begin{corollary}\label{G^s} Let $G$ be a finitely generated group such that for each $g\in G$ there exists $n\geq1$ with the property that $E_n(g)$ satisfies $max$ . Then each term of the derived series of $G$ is finitely generated.
\end{corollary}

We will also require the following lemma.
\begin{lemma}\label{nilinili} Let $G$ be a group such that $G'$ is nilpotent and let $N$ be a normal subgroup of $G$. Suppose that the elements $x,y\in G$ are both Engel in the subgroups $N\langle x\rangle$ and $N\langle y\rangle$, respectively. Then their product $xy$ is Engel in the subgroup $N\langle xy\rangle$.
\end{lemma}
\begin{proof} Set $C_0=1$ and $C_{i+1}/C_i=C_{N/C_i}(G'C_i/C_i)$ for $i=0,1,\dots$. Thus, $C_0\leq C_1\leq\dots$ is a series in $N$ all of whose factors  centralize $G'$. Since $G'$ is nilpotent, there is a number $s$ such that $C_s=N$. If $s=0$, then $N$ is trivial and there is nothing to prove. So we assume that $s\geq1$. Arguing by induction on $s$ assume that the lemma holds for the group $G/C_1$. Thus, there is a number $j$ such that $[g,{}_j\,xy]\leq C_1$ for each $g\in N$. Therefore it is sufficient to prove the lemma with $C_1$ in place of $N$. Hence, without loss of generality we assume that $[N,G']=1$. Let ${\bar G}=G/C_G(N)$. The group ${\bar G}$ naturally acts on $N$ and we can view $N$ as a subgroup in the semidirect product of $N$ by ${\bar G}$. We note that $N\langle xy\rangle$ is nilpotent if and only if so is $N\langle{\bar x}{\bar y}\rangle$. Since ${\bar G}$ is abelian, both subgroups $N\langle{\bar x}\rangle$ and $N\langle{\bar y}\rangle$ are normal in $N{\bar G}$. Moreover both are nilpotent. It follows that their product is nilpotent, too. The lemma follows.
\end{proof}

\section{On soluble groups of automorphisms of polycyclic groups}

Malcev proved that if $\Gamma$ is a soluble group of automorphisms of a polycyclic-by-finite group, then $\Gamma$ is polycyclic \cite{malcev}. In fact a more specific information about $\Gamma$ can be deduced. The aim of this short section is to prove the following proposition. The proof given here was suggested by Dan Segal.

\begin{proposition}\label{polyc} Let $N$ be a polycyclic-by-finite group with $h(N)=h$ and let $\Gamma$ be a soluble group of automorphisms of $N$. Then $h(\Gamma)<h^2+2h$. 
\end{proposition}

First, we require a lemma.

\begin{lemma}\label{yaya} Let $\alpha$ be an automorphism of a group $G$ and suppose that $\alpha$ centralizes a normal subgroup $N\leq G$ of finite index $m$. Then $\alpha^{m!}$ is an inner automorphism.
\end{lemma}
\begin{proof} Since the group of automorphisms of $G/N$ has order dividing $(m-1)!$, it follows that $\beta=\alpha^{(m-1)!}$ stabilizes the series $1\leq N\leq G$. Let $\beta^G=\{\beta_1,\dots,\beta_s\}$. Note that the automorphisms $\beta_1,\dots,\beta_s$ commute. Therefore $\prod_{i=1}^{s}\beta_i$ centralizes $G$. Further, for every $i$ the element $\beta{\beta_i}^{-1}$ belongs to $N$. Let $K=G\langle\alpha\rangle$. Write $$\beta^s=\prod_{i=1}^{s}\beta_i\cdot\prod_{i=1}^{s}\beta{\beta_i}^{-1}\in C_K(G)N.$$ We see that $\beta^s$ is an inner automorphism of $G$ (induced by an element of $N$). Since $s$ is a divisor of $m$, we conclude that $\beta^m$ is an inner automorphism of $G$. It remains to note that $\beta^m=\alpha^{m!}$.
\end{proof}

The proof of Proposition \ref{polyc} will use the concept of plinth. A plinth of a group $G$ is a non-trivial finitely generated free abelian
normal subgroup $A$ containing no non-trivial subgroup of lower rank which is normal in any subgroup of finite index in $G$. Thus $G$ and all its subgroups of finite index act rationally irreducibly on $A$.

\begin{proof}[Proof of Proposition \ref{polyc}] If $h(N)=0$, then $h(\Gamma)=0$ so we will assume that both groups $N$ and $\Gamma$ are infinite. We write $G$ for the product of $N$ by $\Gamma$. By \cite[Exercise 1C9]{segal} $N$ contains a characteristic infinite free abelian subgroup. Therefore $N$ also contains a plinth $N_1$ of some normal subgroup $G_1$ of finite index in $G$ (see \cite[Theorem 7.1.10]{infsol}). If $(N\cap G_1)/N_1$ is infinite, we repeat the above and find a plinth $N_2/N_1$ contained in $(N\cap G_1)/N_1$ of some normal subgroup $G_2/N_1$ of finite index in $G_1/N_1$. Continuing the process we find a series $$1<N_1<\dots<N_s$$ and a subgroup $G_s$ of finite index in $G$ such that $N_s$ has finite index in $N\cap G_s$, the subgroups $N_i$ are normal in $G_s$ and each factor $N_{i+1}/N_i$ is a plinth for $G_s/N_i$.

Set $K=G_s\cap\Gamma$. It is clear that $h(N_s)=h$ and $h(K)=h(\Gamma)$ so it is sufficient to show that $h(K)$ is at most $h^2+2h$. Set $\Delta=C_K(N_s)$, $\Delta_1=C_K(N_1)$, $\Delta_2=C_K(N_s/N_1)$, and $\Delta_0=\Delta_1\cap\Delta_2$. 

Let $h_1=h(N_1)$. The group $K/\Delta_1$ faithfully acts on $N_1$ and so $K/\Delta_1$ embeds in $GL(h_1,\Bbb Z))$. By \cite[Theorem 3.1.8]{infsol} $K/\Delta_1$ is abelian-by-finite. Now we use the fact that rationally irreducible abelian subgroups of $GL(n,\Bbb Z))$ have torsion-free rank at most $n-1$ by the Dirichlet Unit Theorem and conclude that $h(K/\Delta_1)\leq h_1-1$. 

Set $h_2=h-h_1$. The group $K/\Delta_2$ naturally acts on $N_s/N_1$. Arguing by induction on $h$ we can assume that $h(K/\Delta_2)\leq {h_2}^2+2h_2$. Further, the group $\Delta_0/\Delta$ acts faithfully on $N_s$ and stabilizes the series $1\leq N_1\leq N_s$. By \cite[Proposition 1B11]{segal} $\Delta_0/\Delta$ is isomorphic to a subgroup of $Der(N_1,N_s/N_1)$ which is of rank $h_1h_2$.

Finally, note that if $[N:N_s]=m$, by Lemma \ref{yaya} $\Delta^{m!}$ embeds in the group of inner automorphisms of $N$ and therefore $h(\Delta)\leq h$. Thus,
$$h(K)\leq h(K/\Delta_1)+h(K/\Delta_2)+h(\Delta_0/\Delta)+h(\Delta)$$
$$\leq h_1-1+{h_2}^2+2h_2+h_1h_2+h<h^2+2h.$$
\end{proof}

\section{The main theorem}
It is also easy to see that any orderable group is torsion-free. Moreover, if $x,y$ are elements of an orderable group such that $[x,y^m]=1$ for some $m\geq1$, then $x$ and $y$ commute \cite[Lemma 2.5.1 $(i)$]{BMR}. The class of orderable groups is closed under taking subgroups but a quotient of an orderable group is not necessarily orderable \cite[Section 2.1]{BMR}. A subgroup $C$ of an ordered group $(G,\leq)$ is called {\it convex} if $x\in C$ whenever $1\leq x\leq c$ for some $c\in C$. Obviously $\{1\}$ and $G$ are convex subgroups of $G$; and, if $C$ is a convex subgroup, then every conjugate of $C$ is convex. It is also clear that all convex subgroups of an ordered group form, by inclusion, a totally ordered set, which is closed under intersection and union. If $C$ and $D$ are convex subgroups of an ordered group $G$, with $C<D$, and there is not a convex subgroup $H$ of $G$ such that $C<H<D$, we say that the pair $(C,D)$ is a {\it convex jump} in $G$. Orders on a group $G$ in which $\{1\}$ and $G$ are the only convex subgroups are very well known. By a result of H${\rm \ddot{o}}$lder \cite[Theorem\ 1.3.4]{BMR}, a group $G$ with such an order is order-isomorphic to a subgroup of the additive group of the real numbers under the natural order. This implies that, if $(C,D)$ is a convex jump of an ordered group, then $C$ is normal in $D$ and $D/C$ is abelian \cite[Lemma 1.3.6]{BMR}. f $x,y$ are elements of an orderable group such that $[x,y^m]=1$ for some $m\geq1$, then $x$ and $y$ commute \cite[Lemma 2.5.1 $(i)$]{BMR}. The following lemma can be easily deduced from the fact that in an orderable group $[x,y^m]=1$ for some $m\geq1$ implies that $[x,y]=1$.

\begin{lemma}\label{zi}
Let $G$ be an orderable group having a nilpotent subgroup $N$ of finite index. Then $G$ is nilpotent of the same class as $N$.
\end{lemma}

\begin{lemma}\label{cn}
Let $G$ be an orderable group in which for each $x$ there exists $n$ such that $E_n(x)$ satisfies $max$. Then each convex subgroup in $G$ is normal.
\end{lemma}

\begin{proof} Suppose that $C$ is convex and not normal in $G$. Since convex subgroups form a chain, we have either $C^x<C$ or $C<C^x$ for some $x\in G$. Without loss of generality assume that $C<C^x$ and let $c^x\in C^x\backslash C$ for a suitable $c\in C$. Then $C^{x^i}<C^{x^{i+1}}$ for any integer $i$. Moreover, by Lemma \ref{x^y}, the subgroup $\langle c\rangle^{\langle x\rangle}$ is finitely generated so that $\langle c\rangle^{\langle x\rangle}=\langle c^{x^{i_1}},\ldots,c^{x^{i_k}}\rangle$ where $i_1,\ldots,i_k$ are integers. We may assume $i_1<i_2<\ldots<i_k$. It follows that $\langle c\rangle^{\langle x\rangle}\leq C^{x^{i_k}}$. Hence $c^{x^{i_{k}+1}}\in C^{x^{i_k}}$ and therefore $c^x\in C$, a contradiction.
\end{proof}

\begin{lemma}\label{nilpo} Let $n,h\geq1$. Let $G$ be an orderable group in which $E_n(x)$ is polycyclic with $h(E_n(x))\leq h$. Then $G'$ is nilpotent with $(h,n)$-bounded class.
\end{lemma}
\begin{proof} It is sufficient to establish the result under the additional hypothesis that $G$ is finitely generated. Thus, assume that $G$ is finitely generated. We know that the convex subgroups in $G$ are normal and let $C$ be a convex subgroup such that $G/C$ is soluble. Since by Corollary \ref{G^s} all terms of the derived series of $G$ are finitely generated, it follows that $G/C$ has finite rank and therefore, by \cite[Theorem 3.3.1]{BMR}, the derived group $(G/C)'$ is nilpotent. We conclude that the image of $E_n(x)$ in $G/C$ is nilpotent for each $x\in G$. Hence, each element of $(G/C)'$ is Engel and so, by Corollary \ref{ni3}, there is an $h$-bounded number $n_0$ such that each element of $(G/C)'$ is $n_0$-Engel. A result of Zelmanov \cite{zelma} states that a nilpotent torsion-free $n_0$-Engel group is nilpotent with bounded nilpotency class (see also \cite{bume}). In particular, we deduce that $G/C$ has $n$-bounded derived length, say $d$. 

Let $R$ be the intersection of all (normal) convex subgroups $N$ of $G$ such that $G/N$ is soluble. The above argument shows that $G^{(d+1)}\leq R$. Since all terms of the derived series of $G$ are finitely generated, we conclude that $R$ is finitely generated, too. If $R\neq1$, among the convex subgroups properly contained in $R$ we can choose a maximal one, say $D$. It follows that $R/D$ is abelian and so $G/D$ is soluble. This is a contradiction since $R$ is the intersection of all convex subgroups $N$ of $G$ such that $G/N$ is soluble. The conclusion is that $R=1$  and $G$ is soluble with derived length at most $d$. Again we observe that $G$ has finite rank whence $G'$ is nilpotent with $(h,n)$-bounded class.
\end{proof}

We are now ready to complete the proof of our main theorem which we restate here for the reader's convenience.
\medskip

\noindent{\sc Theorem.} {\em Let $h,n$ be positive integers and $G$ an orderable group in which $E_n(x)$ is polycyclic with $h(E_n(x))\leq h$ for every $x\in G$. There are $(h,n)$-bounded numbers $h^*$ and $c^*$ such that $G$ has a finitely generated normal nilpotent subgroup $N$ with $h(N)\leq h^*$ and $G/N$ nilpotent of class at most $c^*$. }
\medskip
\begin{proof} Choose an arbitrary element $x\in G$. Since the Hirsch length of a subgroup of infinite index of a polycyclic group is strictly smaller than that of the group, it follows that in the series $$E_n(x)\geq[E_n(x),x]\geq[E_n(x),x,x]\geq\dots$$ at most $h$ terms $[E_n(x),{}_i\,x]$ have infinite index in $[E_n(x),{}_{i-1}\,x]$. In view of Lemma \ref{nilpo}, $E_n(x)$ is nilpotent. The element $x$ naturally acts on $E_n(x)$ by conjugation and Lemma \ref{ni2} shows that if, for some $i\geq1$, the subgroup $[E_n(x),{}_i\,x]$ has finite index in $[E_n(x),{}_{i-1}\,x]$, then $[E_n(x),{}_{i+s}\,x]$ has finite index in $[E_n(x),{}_{i-1}\,x]$ for any $s\geq1$. It follows that $[E_n(x),{}_{h+s}\,x]$ has finite index in $[E_n(x),{}_{h}\,x]$ for any $s\geq1$. For all $x\in G$ set $U(x)=[E_n(x),{}_{h}\,x]$.
It follows that an element $x$ is an Engel element if and only if $U(x)=1$ and each Engel element in $G$ is $(n+h)$-Engel.

We know from Lemma \ref{nilpo} that $G'$ is nilpotent with $(h,n)$-bounded class. By Lemma \ref{bb} there exists a bounded number $f$ such that for each element $x\in G$ we have $\gamma_f(\langle x,G'\rangle)\leq E_{n+h}(x)\leq U(x)$. Observe that for each $i$ the subgroup $[U(x),{}_i\,x]$ is contained in $\gamma_{n+h+i+1}(\langle x,G'\rangle)$. Since by our assumptions for each positive $i$ the subgroup $[U(x),{}_i\,x]$ has finite index in $U(x)$, we conclude that also $\gamma_{f+i}(\langle x,G'\rangle)$ has finite index in $U(x)$. Em particular, $h(\gamma_{f+i}(\langle x,G'\rangle))=h(U(x))$. In the sequel we will use without mentioning explicitly that all subgroups of the form $\gamma_j(\langle x,G'\rangle)$ are normal in $G$.

Now choose $a\in G$ such that the Hirsch length $h_0$ of $\gamma_f(\langle a,G'\rangle)$ is as big as possible. If $h_0=0$, then $U(x)=1$ for each $x\in G$ and so all elements of $G$ are $(n+h)$-Engel. Hence, by Zelmanov's result \cite{zelma}, $G$ is nilpotent with bounded class. Therefore we will assume that $h_0\geq1$. Let $b\in C_G(\gamma_f(\langle a,G'\rangle))$. Recall that in an orderable group $[x,y^m]=1$ implies that $[x,y]=1$. Since $b$ centralizes $\gamma_f(\langle a,G'\rangle)$, which is a subgroup of finite index in $U(a)$, it follows that $b$ centralizes $U(a)$. Set $S=\gamma_f(\langle b,G'\rangle)$. Define $T=1$ if $S$ is abelian and $T=Z(S)$ (the center of $S$) otherwise. Note that if $U(b)\neq1$, then $S/T$ is infinite. Since $b$ centralizes $U(a)$, we deduce that $U(b)\cap U(a)\leq T$. Further, observe that $a$ acts on $S/T$ as an $(n+h)$-Engel element. Since $b\in C_G(\gamma_f(\langle a,G'\rangle))$, the action of $ab$ on $\gamma_f(\langle a,G'\rangle)$ is the same as that of $a$. Therefore the subgroups of the form $[\gamma_f(\langle a,G'\rangle),{}_iab]$ have finite index in $\gamma_f(\langle a,G'\rangle)$ for each $i$. It follows that $\gamma_f(\langle ab,G'\rangle)$ intersects $\gamma_f(\langle a,G'\rangle)$ by a subgroup having finite index in both $\gamma_f(\langle ab,G'\rangle)$ and $\gamma_f(\langle a,G'\rangle)$. Therefore the Hirsch length of $\gamma_f(\langle ab,G'\rangle)$ is precisely $h_0$. Since $b$ centralizes a subgroup of finite index in $\gamma_f(\langle ab,G'\rangle)$, we conclude that $b$ centralizes $\gamma_f(\langle ab,G'\rangle)$. Taking into account that $S\cap \gamma_f(\langle a,G'\rangle)\leq T$ we further deduce that $S\cap\gamma_f(\langle ab,G'\rangle)\leq T$. Hence, $ab$ acts on $S/T$ as an $(n+h)$-Engel element. Thus, both $a$ and $ab$ act on $S/T$ as Engel elements. Lemma \ref{nilinili} now shows that also $b$ acts on $S/T$ as an Engel element. We know that $[S,{}_i\,b]$ has finite index in $S$ for every $i$. Therefore we now deduce that $S=1$, that is, $b$ is an Engel element in $G$. Recall that $b$ was chosen in $C_G(\gamma_f(\langle a,G'\rangle))$ arbitrarily. Thus, each element of $C_G(\gamma_f(\langle a,G'\rangle))$ is $(n+h)$-Engel in $G$.

Let $F$ be the Fitting subgroup of $G$. We know that $F$ consists of $(n+h)$-Engel elements. Therefore, by Zelmanov's result \cite{zelma}, $F$ is nilpotent with bounded class. Moreover both $G'$ and $C_G(\gamma_f(\langle a,G'\rangle))$ are contained in $F$. Further, using Lemma \ref{zi} we note that $G/F$ is torsion-free. The group $G/C_G(\gamma_f(\langle a,G'\rangle))$ faithfully acts on $\gamma_f(\langle a,G'\rangle)$. By Proposition \ref{polyc} $G/C_G(\gamma_f(\langle a,G'\rangle))$ has $h$-bounded Hirsch length. Therefore $G/F$ is abelian with $h$-bounded number of generators. Write $G=\langle F,a_1,\dots,a_k\rangle$, where $a_1,\dots,a_k$ are (boundedly many) generators of $G$ modulo $F$. For $i=1,\dots,k$ set $G_i=\langle F,a_i\rangle$. All the subgroups $G_i$ are normal in $G$ since $F$ contains $G'$. According to Lemma \ref{bb} there is an $(h,n)$-bounded number $f_0$ such that Then $\gamma_{f_0}(G_i)\leq E_n(a_i)$. Our hypotheses imply that $\gamma_{f_0}(G_i)$ is polycyclic with Hirsch length at most $h$ for each $i=1,\dots,k$. Let $N$ be the subgroup generated by all $\gamma_{f_0}(G_i)$. It follows that $N$ is polycyclic with Hirsch length at most $kh$. Further, $G/N$ is a product of $k$ normal subgroups $G_iN/N$, each of which is nilpotent of class at most $f_0-1$. It follows that $G/N$ is nilpotent of class at most $kf_0-k$. Thus, we can take $h^*=kh$ and $c^*=kf_0-k$.
\end{proof}


\begin{thebibliography}{10}

\bibitem{BMR} R. Botto Mura and A.\,H. Rhemtulla, {\it Orderable groups}, Lecture Notes in Pure and Applied Mathematics, Marcel Dekker, Inc., New York - Basel, 1977.
\bibitem{bume} R.\,G. Burns and Y. Medvedev, A note on Engel groups and local nilpotence, J. Austral Math. Soc. Ser. A {\bf 64} (1998), no. 1, 92--100.
\bibitem{khushu1} E. I. Khukhro and P. Shumyatsky, Almost Engel finite and profinite groups, Int. J. Algebra Comput., {\bf  26} (2016), 973--983.
 2001.
\bibitem{khushu2} E. I. Khukhro and P. Shumyatsky, Almost Engel compact groups. J. Algebra (2017), doi:10.1016/j.jalgebra.2017.04.021.
\bibitem{kr2} Y. Kim and A.\,H. Rhemtulla, {\it Groups with ordered structures}, Groups-Korea '94 (Pusan), 199--210, de Gruyter, Berlin, 1995.
\bibitem{infsol}  J.\,C. Lennox, D.\,J.\,S. Robinson,
\textit{The theory of infinite soluble groups}, Clarendon Press, Oxford, 2004.
\bibitem{LM} P. Longobardi and M. Maj, {\it On some classes of orderable groups}, Rend. Sem. Mat. Fis. Milano {\bf 68} (1998), 203--216, 2001.
\bibitem{malcev} A.\,I. Malcev, On some classes of infinite soluble groups. (Russian). Mat. Sbornik N.S., {\bf 28(70)} (1951), 567�588;  On some classes of infinite soluble groups. Amer. Math. Soc. Translations, {\bf 2} (1956), 1�21.
\bibitem{segal} D. Segal, Polycyclic Groups, Cambridge Tracts in Mathematics, {\bf 82}, Cambridge University Press, Cambridge, 1983.
\bibitem{mona} P. Shumyatsky, Almost Engel linear groups. Monatsh Math (2017), doi:10.1007/s00605-017-1062-x.
\bibitem{zelma} E. Zelmanov, Some problems in the theory of groups and Lie algebras (Russian), Mat. Sb. {\bf 180} (1989), 159--167; translation in Math. USSR-Sb. {\bf 66} (1990), 159--168.

\end{thebibliography}
\end{document}